\documentclass[a4paper,12pt]{amsart}

\title[Inhomogeneous multinomial measures] {\bf Multifractal analysis of some
  inhomogeneous multinomial measures with distinct analytic Olsen's
  $b$ and $B$ functions} \author{Shuang Shen} \date{November 2013}

\usepackage{latexsym,amsmath,amsfonts,mathrsfs}
\usepackage{amsthm}
\usepackage{amssymb}
\usepackage{amstext}
\usepackage{cite}

\usepackage{mathrsfs}
\usepackage{amsfonts}
\usepackage{galois}
\usepackage{amsopn}
\usepackage[dvips]{color}
\DeclareMathOperator{\Dim}{\mathrm{Dim}}

\DeclareMathOperator*{\esssup}{\mathrm{ess\,sup}}

\newcommand{\ball}{\mathrm{B}}
\newtheorem{thm}{Theorem}
\newtheorem{pro}[thm]{Proposition}
\newtheorem{lemma}[thm]{Lemma}
\newtheorem{cor}[thm]{Corollary}
\theoremstyle{remark}
\newtheorem{remark}{Remark}
\newtheorem{definition}{Definition}
\begin{document}

\thanks{The research is supported by CNSF: No. 11271223.}

\begin{abstract}
  Inhomogeneous multinomial measures on the mixed symbolic spaces and
  the real line are given. By counting the zeros of the corresponding
  generalized Dirichlet polynomials, one obtains a probability measure
  whose Olsen's functions $b$ and $B$ are analytic and their graphs
  differ except at two points where they are tangent. Also,
  interpretations of the Legendre transform of $b$ and $B$ are given
  in terms of dimensions.

\end{abstract}

\maketitle

\section{Introduction}
In this paper, we want to describe a probability measure whose
Olsen's $b$ and $B$ functions are analytic and their graphs differ
except at two points where they are tangent. This is motivated by an
example in \cite{Ben}. We work first in the symbolic space and we
construct a class of inhomogeneous multinomial measures, some of which
have good properties such that we know exactly the numbers and orders
of zeros of the corresponding generalized Dirichlet polynomials. Then
we compute the Hausdorff and packing dimensions of the level sets of
local H\"{o}lder exponents of the measure, which turn out to be the
Legendre transform of the Olsen functions. At last, we consider the
image measure onto the real line applying a Gray code, so to obtain a
doubling measure.

During the completion of the paper, we just happen to know the work by
Barral \cite{Bar2} in which he proves that for two functions under
certain conditions, there exists a compactly supported Borel positive
and finite measure $\mu$ on $\mathbb{R}^d$ such that the free energy
functions associated with $\mu$ are just the two given functions. Also
the constructed measure is exact dimensional.

\section{Notations and definitions}
\subsection{The Olsen measures}{\ }

We deal with a metric space $(\mathbb{X},d)$ possessing the
Besicovitch property:
\medskip

\textsl{There exists a constant $C_\ball\in \mathbb{N}$ such that one can
  extract $C_\ball$ countable families $\{\{\ball_{j,k}\}_k\}_{1\leq j\leq
    C_\ball}$ from any collection $\mathscr{B}$ of balls in $X$ so that
\begin{enumerate}
\item[--] $\bigcup_{j,k}\ball_{j,k}$ contains the centers of the elements of
$\mathscr{B}$,
\item[--] for any $j$ and $k\neq k'$, $\ball_{j,k}\cap \ball_{j,k'}=\emptyset$.
\end{enumerate}}
\medskip

This is a theorem by Besicovitch that an Euclidean space fulfils
this condition. It is obvious that an ultrametric space also has
this property.

Let $\mu$ be a Borel probability measure on $\mathbb{X}$. If $E$ is
a nonempty subset of $\mathbb{X}$ and if $q, t\in \mathbb{R}$ and
$\delta >0$, we introduce the quantities:
\begin{multline*}
\overline{\mathscr{H}}_{\mu,\delta}^{q,t}(E)=\inf\left\{\sum_i^\ast r_i^t
\mu(\ball(x_i,r_i))^q \ :\  \right .\\
\left .(\ball(x_i,r_i))_i \textrm{ centered
  $\delta$-covering of }E\vphantom {\sum_i^\ast}\right\},
\end{multline*}

$$\overline{\mathscr{H}}_{\mu}^{q,t}(E)=\sup_{\delta>0}
\overline{\mathscr{H}}_{\mu,\delta}^{q,t}(E),$$

$$\mathscr{H}_{\mu}^{q,t}(E)=\sup_{F\subset
  E}\overline{\mathscr{H}}_{\mu}^{q,t}(F);$$
and
$$\overline{\mathscr{P}}_{\mu,\delta}^{q,t}(E)=\sup\left\{\sum_i^\ast r_i^t
 \mu(\ball(x_i,r_i))^q \ :\  (\ball(x_i,r_i))_i \textrm{ $\delta$-packing of }E\right\},$$

$$\overline{\mathscr{P}}_{\mu}^{q,t}(E)=
\inf_{\delta>0}\overline{\mathscr{P}}_{\mu,\delta}^{q,t}(E),$$

$$\mathscr{P}_{\mu}^{q,t}(E)=\inf_{E\subset \cup_i E_i}\sum_i
\overline{\mathscr{P}}_{\mu}^{q,t}(E_i).$$

The star means that we omit in the summation the terms which are
obviously infinite (i.e. zero raised to a negative power). However, we
will leave the star out for simplicity. $\ball(x_i,r_i)$ stands for
the open ball centered at point $x_i\in X$ with radius $r_i$, and we
denote by $\ball_i$ for short in the context.  When we say
$(\ball(x_i,r_i))_i$ is a $\delta$-covering of $E$, we mean that
$\bigcup \ball(x_i,r_i)\supseteq E$ and for any $i$, $r_i< \delta$.
When we say $(\ball(x_i,r_i))_i$ is a centered $\delta$-covering of
$E$, we mean it is not only a $\delta$-covering of $E$, but also for
any $i$, $x_i\in E$. At last, when we say $(\ball(x_i,r_i))_i$ is a
$\delta$-packing of $E$, we mean that for any $i$, $x_i\in E$, $r_i<
\delta$ and for any $j\neq k$, $\ball(x_j,r_j)\cap
\ball(x_k,r_k)=\emptyset$.

We can see that respectively, the functions
$\mathscr{H}_{\mu}^{q,t}$, $\mathscr{P}_{\mu}^{q,t}$ and
$\overline{\mathscr{P}}_{\mu}^{q,t}$ are multifractal extensions of
the centered Hausdorff measure $\mathscr{H}^t$, the packing measure
$\mathscr{P}^t$, and the packing premeasure
$\overline{\mathscr{P}}^t$.

\subsection{The Olsen's functions}{\ }

The functions $\mathscr{H}_{\mu}^{q,t}$, $\mathscr{P}_{\mu}^{q,t}$
and $\overline{\mathscr{P}}_{\mu}^{q,t}$ induce dimensions to each
subset $E$ of $\mathbb{X}$. They are defined by
$$b_{\mu,E}(q)=\sup\left\{s\ :\ \mathscr{H}_\mu^{q,s}(E)=
\infty\right\}=\inf\left\{s\ :\ \mathscr{H}_\mu^{q,s}(E)=0\right\},$$
$$B_{\mu,E}(q)=\sup\left\{s\ :\ \mathscr{P}_\mu^{q,s}(E)=
\infty\right\}=\inf\left\{s\ :\ \mathscr{P}_\mu^{q,s}(E)=0\right\},$$
$$\tau_{\mu,E}(q)=\sup\left\{s\ :\ \overline{\mathscr{P}}_\mu^{q,s}(E)=
\infty\right\}=\inf\left\{s\ :\ \overline{\mathscr{P}}_\mu^{q,s}(E)=0\right\}.$$

They are multifractal extensions of the Hausdorff dimension
$\dim E$, the packing dimension $\Dim E$, and the packing
predimension $\Delta E$.

Denote by $S_\mu$ the support of $\mu$. For simplicity, we will
write $b_\mu=b_{\mu,S_\mu}$, $B_\mu=B_{\mu,S_\mu}$, $\tau_\mu=\tau_{\mu,S_\mu}$. And if the measure $\mu$ is clear in the context, we will omit $\mu$ and write $b,B,\tau$ respectively. They satisfy the
following properties (see \cite{Ols}):

\medskip
$1^\circ$ $b(0)=\dim S_\mu$, $B(0)=\Dim S_\mu$,
$\tau(0)=\Delta S_\mu$.

$2^\circ$ $b(1)=B(1)=\tau(1)=0$.

$3^\circ$ $b\leq B\leq \tau$.

$4^\circ$ $b$ is decreasing and $B$ and $\tau$ are convex and
decreasing.
\medskip

There is another way to describe $\tau_{\mu,E}$. Fix $\lambda<1$ and
define:
$$\widetilde{\mathscr{P}}_{\mu,\delta}^{q,t}(E)=\sup\left\{\sum_i r_i^t
 \mu(\ball_i)^q \ :\  (\ball_i)_i \textrm{ packing of }E\ \textrm{with }\lambda\delta<r_i\leq\delta\right\},$$

$$\widetilde{\mathscr{P}}_{\mu}^{q,t}(E)=
\limsup_{\delta\rightarrow
0}\widetilde{\mathscr{P}}_{\mu,\delta}^{q,t}(E),$$

$$\widetilde{\tau}_{\mu,E}(q)=\sup\left\{s\ :\ \widetilde{\mathscr{P}}_\mu^{q,s}(E)=
\infty\right\}=\inf\left\{s\ :\ \widetilde{\mathscr{P}}_\mu^{q,s}(E)=0\right\}.$$

\begin{lemma}\label{l1}  \emph{(see \cite{Ben2})}
$\widetilde{\tau}_{\mu,E}=\tau_{\mu,E}.$ So for any $\lambda<1$, one
has
\begin{multline*}
\tau_{\mu,E}(q)=\limsup_{\delta\rightarrow
0}\frac{-1}{\log \delta}\log\sup\left\{\sum_i
\mu(\ball_i)^q\ :\  \right .\\
 \left .(\ball_i)_i \textrm{ packing of }E\ \textrm{with }\lambda\delta<r_i\leq\delta\vphantom {\sum_i}\right\}.
\end{multline*}

\end{lemma}

\subsection{The multifractal formalism}{\ }

\subsubsection{Level sets of local H\"{o}lder exponents}
Let $\mu$ be a measure on $\mathbb{X}$. For $x\in \mathbb{X}$, we
define the local dimensions or local H\"{o}lder exponents of the
measure $\mu$ at point $x$ by
$$ \overline{\alpha}_\mu(x)=\limsup_{r\rightarrow
  0}\frac{\log\mu(\ball(x,r))}{\log r} \textrm{ and
}\underline{\alpha}_\mu(x)=\liminf_{r\rightarrow
  0}\frac{\log\mu(\ball(x,r))}{\log r}. $$

Then for $\alpha,\beta\in\mathbb{R}$, we introduce the sets:
$$\overline{X}_\mu(\alpha)=
\left\{x\in S_\mu\ :\ \,\overline{\alpha}_\mu(x)\leq\alpha\right\},$$
$$\underline{X}_\mu(\alpha)=
\left\{x\in S_\mu\ :\ \,\underline{\alpha}_\mu(x)\geq\alpha\right\},$$
$$X_\mu(\alpha,\beta)=\underline{X}_\mu(\alpha)\cap\overline{X}_\mu(\beta),$$
$$X_\mu(\alpha)=X_\mu(\alpha,\alpha).$$

As previously, we will omit $\mu$ if the measure is clear.
\subsubsection{Some consequences}
As is known, the multifractal formalism aims at giving expressions
of the dimension of the level sets of local H\"{o}lder exponents of
measure $\mu$ in terms of the Legendre transform of some free
``energy'' function (see \cite{Fri}, \cite{Hal}).

Let $f^\ast(x)=\inf_y (xy+f(y))$ denote the Legendre transform of the
function $f$. Olsen proved the following general estimation, claiming
that the Legendre transform of the dimension functions $b$ and $B$ are
upper bounds of the dimensions of level sets.

\begin{thm}\label{t2}  \emph{(see \cite{Ols})}
Let $\mu$ be a probability measure on $\mathbb{X}$. Define
$\underline{a}=\sup_{q>0}-\frac{b(q)}{q}$ and
$\overline{a}=\inf_{q<0}-\frac{b(q)}{q}$. For all $\alpha\in
(\underline{a},\overline{a})$, we have
$$\dim X(\alpha)\leq b^\ast(\alpha),$$
$$\Dim X(\alpha)\leq B^\ast(\alpha).$$
\end{thm}

\begin{definition}
If $B'(q)$ exists and if $\dim X(-B'(q))=\Dim X(-B'(q))=
b^\ast(-B'(q))=B^\ast(-B'(q))$, we say that the measure $\mu$ obeys
the multifractal formalism at point $q$.

\end{definition}

It always needs some extra conditions to obtain a minoration for the
dimensions of level sets.

\begin{lemma}\label{l3}  \emph{(see \cite{Ben2})}
Let $\mu,\nu$ be two probability measures on $\mathbb{X}$. Fix $\lambda<1$. Set
\begin{multline*}
\varphi(x)=\limsup_{\delta\rightarrow
0}\frac{-1}{\log \delta}\log\sup\left\{\sum_i
 \mu(\ball_i)^x\nu(\ball_i)\ :\  \right .\\
 \left .(\ball_i)_i \textrm{ packing of }S_\mu\ \textrm{with }\lambda\delta<r_i\leq\delta\vphantom {\sum_i}\right\}.
\end{multline*}

Assume that
$\varphi(0)=0$, $\nu (S_\mu)>0$, and that $\varphi'(0)$ exists. Let $E=X_\mu(-\varphi'(0))$, then one has
$$\dim E\geq \esssup_{x\in E,\nu} \liminf_{r\rightarrow
0} \frac{\log
\nu(B(x,r))}{\log r},$$
$$\Dim E\geq \esssup_{x\in E,\nu} \limsup_{r\rightarrow
0} \frac{\log
\nu(B(x,r))}{\log r}.$$

\end{lemma}

Ben Nasr, Bhouri and Heurteaux gave a sufficient condition and a
necessary condition for a valid multifractal formalism, showing that
the knowledge of the so-called Gibbs measure is quite unnecessary.
That is

\begin{thm}\label{t4}  \emph{(see \cite{Ben})}
Let $\mu$ be a probability measure on $\mathbb{X}$ and
$q\in\mathbb{R}$. Suppose that $B'(q)$ exists.

(\romannumeral1) If $\mathscr{H}^{q,B(q)}_\mu(S_\mu)>0$, then
$$\dim X(-B'(q))=\Dim X(-B'(q))=
b^\ast(-B'(q))=B^\ast(-B'(q)).$$

(\romannumeral2) Conversely, if $\dim X(-B'(q))\geq
B^\ast(-B'(q))$, then $b(q)=B(q)$.
\end{thm}

From the second part, when $B'(q)$ exists, $b(q)=B(q)$, known as the
Taylor regularity condition, is the necessary condition for a valid
multifractal formalism.

\section{Some measures on symbolic spaces}

\subsection{The symbolic spaces}{\ }

Let $c\geq 2$, $\mathscr{A}=\{0,1,\cdots,c-1\}$. We consider
$\mathscr{A}^\ast=\bigcup_{n\geq 0}\mathscr{A}^n$, the set of all
finite words on the $c$-letter alphabet $\mathscr{A}$.

If $w=\varepsilon_1\cdots\varepsilon_n$ and
$v=\varepsilon_{n+1}\cdots\varepsilon_{n+p}$, denote by $w\cdot v$
(or simply by $wv$ if it is not ambiguous) the word $\varepsilon_1
\cdots\varepsilon_{n+p}$. With this operation, $\mathscr{A}^\ast$ is a
monoid whose identity element is the empty word $\epsilon$. If a
word $v$ is a prefix of the word $w$, we write $v\prec w$. This
defines an order on $\mathscr{A}^\ast$ and endowed with this order,
$\mathscr{A}^\ast$ becomes a tree whose root is $\epsilon$.  At last,
the length of a word $w$ is denoted by $|w|$. If $w$ and $v$ are two
words, $w\wedge v$ stands for their largest common prefix. It is well
known that the function $d(w,v)=c^{-|w\wedge v|}$ defines an
ultrametric distance on $\mathscr{A}^\ast$.

The completion of $(\mathscr{A}^\ast,d)$ is a compact space which is
the disjoint union of $\mathscr{A}^\ast$ and
$\partial\mathscr{A}^\ast$, whose elements can be viewed as infinite
words. Each finite word $w\in\mathscr{A}^\ast$ defines a cylinder
$[w]=\{x\in\partial\mathscr{A}^\ast|w\prec x\}$, which can also be
viewed as a ball. For a Borel measure $\mu$ on
$\partial\mathscr{A}^\ast$, we simply write $\mu([w])=\mu(w)$. Thus we
identify the Borel measure $\mu$ on $\partial\mathscr{A}^\ast$ with a
mapping from $\mathscr{A}^\ast$ to $[0,+\infty]$ so that for any
$w\in\mathscr{A}^\ast$,
$$\mu(w)=\sum_{j\in\mathscr{A}}\mu(w j).$$

Since the diameters of balls in $\partial\mathscr{A}^\ast$ are
$c^{-n}$, one can compute the function $\tau$ in the following way
according to Lemma 1. That is
$$\sum_{w\in\mathscr{A}^n} \mu(w)^q=c^{n\tau_n(q)},$$
$$\tau(q)=\limsup_{n\rightarrow\infty} \tau_n(q).$$
Also, we denote
$$\underline{\tau}(q)=\liminf_{n\rightarrow\infty} \tau_n(q).$$

Now we define the mixed symbolic spaces. Let $c_1,c_2\geq 2$,
$\mathscr{A}_1=\{0,1,\cdots,c_1-1\}$, $\mathscr{A}_2=\{0,1,\cdots,c_2-1\}$
be two alphabets. Let $(T_k)$ be a sequence of integers such that
$$T_1=1, T_k<T_{k+1} \textrm{ and } \lim_{k\rightarrow\infty}
T_{k+1}/T_k=+\infty.$$

Consider the set of infinite words
$$\partial\mathscr{A}_{1,2}^\ast =
\mathscr{A}_1^{T_2-T_1}\mathscr{A}_2^{T_3-T_2}\cdots=\prod_j
X_j,$$
where

\begin{enumerate}
\item[--] if $T_{2k-1}\leq j< T_{2k}$ for some $k$, $X_j=\mathscr{A}_1$,
\item[--] if $T_{2k}\leq j< T_{2k+1}$ for some $k$, $X_j=\mathscr{A}_2$.
\end{enumerate}

We call $\partial\mathscr{A}_{1,2}^\ast$ the mixed symbolic
space with respect to $\{\mathscr{A}_1,\mathscr{A}_2,(T_k)\}$.

Let $N_n$ be the number of integers $j\leq n$ such that
$X_j=\mathscr{A}_1$. We can immediately get that
$$\liminf_{n\rightarrow\infty}\frac{N_n}{n}=0,$$ and
$$\limsup_{n\rightarrow\infty}\frac{N_n}{n}=1.$$

For any two different elements $w,v\in \partial\mathscr{A}_{1,2}^\ast$
with $|w\wedge v|=n$, we define that
$d(w,v)=c_1^{-N_n}c_2^{-(n-N_n)}$. As previously, this defines an
ultrametric distance.

One sees that when $c_1=c_2=c$, the mixed symbolic space becomes
ordinary symbolic space.

\subsection{Inhomogeneous multinomial measures}{\ }

In \cite{Ben}, the authors presented a measure which has an analytic
function $B$ and a linear function $b$, such that the graph of $B$
is tangent to the graph of $b$ at point (1,0). Here, we wish to show
that there exists a measure $\mu$, such that the Olsen functions
$B$ and $b$ both are analytic and their graphs differ except at two
points where they are tangent; what is more, $B'(\mathbb{R})$ and
$b'(\mathbb{R})$ both are intervals of positive length.

We first work on the symbolic spaces and begin with the following
result, giving the so called inhomogeneous multinomial measures.

\begin{thm}\label{t5}
Let $\mathscr{A}_1=\{0,1,\cdots,c_1-1\}$,
$\mathscr{A}_2=\{0,1,\cdots,c_2-1\}$ and let $(T_k)$ be a sequence of
integers such that
$$T_1=1, T_k<T_{k+1} \textrm{ and } \lim_{k\rightarrow\infty}
T_{k+1}/T_k=+\infty.$$ Let $a_i, b_j\in
(0,1)$($i=1,\cdots,c_1,j=1,\cdots,c_2$) satisfying
$a_1+\cdots+a_{c_1}=b_1+\cdots+b_{c_2}=1$. There exists a probability
measure $\mu$ on $\partial\mathscr{A}_{1,2}^\ast$ such that for every
$q\in \mathbb{R}$,
$$B(q)=\sup\{\log_{c_1}(a_1^q+\cdots+a_{c_1}^q),
\log_{c_2}(b_1^q+\cdots+b_{c_2}^q)\},$$
$$b(q)=\inf\{\log_{c_1}(a_1^q+\cdots+a_{c_1}^q),
\log_{c_2}(b_1^q+\cdots+b_{c_2}^q)\}.$$
\end{thm}

To avoid tedious notations, we write the proof with $c_1=c_2=3$. The
reader will realize that the general case can be handled with minor
modifications.

Now $\mathscr{A}_1=\mathscr{A}_2=\mathscr{A}=\{0,1,2\}$. We define the
measure $\mu$ on $\partial\mathscr{A}^\ast$ such that for every
cylinder $[\varepsilon_1 \varepsilon_2\cdots\varepsilon_n]$, one has
$$\mu(\varepsilon_1\cdots\varepsilon_n)=\prod_{j=1}^n p_{j},$$
where

\begin{enumerate}
\item[--] if $T_{2k-1}\leq j< T_{2k}$ for some $k$, $p_j=a_{\varepsilon_j+1}$,
\item[--] if $T_{2k}\leq j< T_{2k+1}$ for some $k$, $p_j=b_{\varepsilon_j+1}$.
\end{enumerate}

To prove the theorem, we first compute the $\tau$ function with
respect to the measure $\mu$.

If $N_n$ is the number of integers $j\leq n$ such that $p_j\in
\{a_1,a_2,a_3\}$, we can easily deduce that
$$\tau_n(q)=\frac{N_n}{n}\log_3(a_1^q+a_2^q+a_3^q)+
(1-\frac{N_n}{n})\log_3(b_1^q+b_2^q+b_3^q).$$ By using that
$\liminf_{n\rightarrow\infty}\frac{N_n}{n}=0$ and
$\limsup_{n\rightarrow\infty}\frac{N_n}{n}=1$, we can conclude
$$\tau(q)=\sup\{\log_3(a_1^q+a_2^q+a_3^q),\log_3(b_1^q+b_2^q+b_3^q)\},$$
$$\underline{\tau}(q)=\inf\{\log_3(a_1^q+a_2^q+a_3^q),
\log_3(b_1^q+b_2^q+b_3^q)\}.$$

\begin{lemma}\label{l6}
Let $\mathscr{A}=\{0,1,2\}$. We can construct a probability
measure $\nu$ on $\partial\mathscr{A}^\ast$ and a
subsequence of integers $(n_k)_{k\geq 1}$, such that
$$\nu(w)\leq \mu(w)^q 3^{-n\underline{\tau}(q)},\textrm{ if } w\in
  \mathscr{A}^n,$$
and for every $\varepsilon> 0$,
$$\nu(w)\leq \mu(w)^q 3^{-n_k(\tau(q)-\varepsilon)}, \textrm{ if } w\in
\mathscr{A}^{n_k} \textrm{ with $k$ large}.$$

\end{lemma}

\begin{proof}
Define a mapping $\nu_n$ from $\mathscr{A}^\ast$ to $[0,+\infty]$
such that for any $x\in\mathscr{A}^m$,
$$\nu_n(x)=\sum_{w\in \mathscr{A}^n,x\prec w}\mu(w)^q
3^{-n\tau_n(q)},\textrm{ if }m\leq n,$$
$$\nu_n(x)=\frac{1}{c^{m-n}}\sum_{w\in \mathscr{A}^n,w\prec x}\mu(w)^q
3^{-n\tau_n(q)},\textrm{ if }m>n.$$

Then it is easy to see that $\nu_n$ is a probability measure on
$\partial\mathscr{A}^\ast$ and for every $w\in\mathscr{A}^n$,
$$\nu_n(w)=\mu(w)^q 3^{-n\tau_n(q)}.$$

Let $w\in \mathscr{A}^n$ and $p>0$. We have
$$\nu_{n+p}(w)=\sum_{x\in \mathscr{A}^p}\nu_{n+p}(w x)=\sum_{x\in
  \mathscr{A}^p}\mu(w x)^q 3^{-(n+p)\tau_{n+p}(q)}.$$

Meanwhile,
$$\frac{\sum_{x\in
  \mathscr{A}^p}\mu(w x)^q}{\mu(w)^q}=\frac{\sum_{z\in
  \mathscr{A}^{n+p}}\mu(z)^q}{\sum_{w\in
  \mathscr{A}^n}\mu(w)^q}=\frac{3^{(n+p)\tau_{n+p}(q)}}{3^{n\tau_n(q)}}.$$

We can conclude that
$$\nu_{n+p}(w)=\mu(w)^q\frac{3^{(n+p)\tau_{n+p}(q)}}{3^{n\tau_n(q)}}
3^{-(n+p)\tau_{n+p}(q)}=\mu(w)^q
3^{-n\tau_n(q)}=\nu_n(w).$$

Let $(n_k)_{k\geq 1}$ be a subsequence such that
$\tau(q)=\lim_{k\rightarrow +\infty}\tau_{n_k}(q)$ and choose $\nu$ as
a weak limit of a subsequence of $\nu_{n_k}$. Observing that
$\underline{\tau}(q)\leq\tau_n(q)$, we obtain that
$$\forall n\geq 1,\forall w\in\mathscr{A}^n, \nu(w)\leq \mu(w)^q
3^{-n\underline{\tau}(q)}.$$

If $\varepsilon>0$
and if $k$ is sufficiently large, we have
$\tau(q)-\varepsilon\leq\tau_{n_k}(q)$. Thus for any
$w\in\mathscr{A}^{n_k}$, $\nu_{n_k}(w)\leq \mu(w)^q
3^{-{n_k}(\tau(q)-\varepsilon)}$, and finally we get $\nu(w)\leq
\mu(w)^q 3^{-n_k(\tau(q)-\varepsilon)}$.

\end{proof}

\begin{lemma}\label{l7}
For any fixed $q$, one has
$$\mathscr{P}^{q,\tau(q)-\varepsilon}_\mu(S_\mu)> 0, \textrm{ for
  every }\varepsilon>0,$$
$$\mathscr{H}^{q,\underline{\tau}(q)+\varepsilon}_\mu(S_\mu)<
+\infty,\textrm{ for every }\varepsilon>0,$$
and
$$\mathscr{H}^{q,\underline{\tau}(q)}_\mu(S_\mu)> 0.$$
\end{lemma}

\begin{proof}
To prove the first inequality, one takes any family of $\{E_i\}$
such that $S_\mu=\cup E_i$ and for each $i$ one computes
$\overline{\mathscr{P}}^{q,\tau(q)-\varepsilon}_\mu(E_i)$.

For any $\delta>0$, for any $x\in S_\mu$, there exists an integer
$n_x$ and a word $w_x\in \mathscr{A}^{n_x}$ such that
$w_x\prec x$, $3^{-n_x}<\delta$ and $\nu(w_x)\leq
\mu(w_x)^q 3^{-n_x (\tau(q)-\varepsilon)}$. When we identify a
finite word with a cylinder as well as a ball, by Besicovitch
property, we can extract from $\{w_x\}_{x\in E_i}$ $C_\ball$
countable families $\{w_{j,k}\}_{1\leq j\leq C_\ball, k\geq 1}$ such
that $\cup_{j,k}w_{j,k}\supseteq E_i$ and for any $j$,
$\{w_{j,k}\}_{k\geq
  1}$ is a $\delta$-packing of $E_i$.

Then one gets
$$\nu^\ast(E_i)\leq \sum_{j,k}\nu^\ast(w_{j,k})\leq
\sum_{j,k}\nu(w_{j,k})\leq \sum_{j,k} \mu(w_{j,k})^q
3^{-|w_{j,k}|(\tau(q)-\varepsilon)},$$ where $\nu^\ast$
stands for the outer measure of $\nu$.

So there exists $j$ such that $\sum_k \mu(w_{j,k})^q
3^{-|w_{j,k}|(\tau(q)-\varepsilon)}\geq
\frac{1}{C_\ball}\nu^\ast(E_i)$. Thus
$$\sum_i\overline{\mathscr{P}}^{q,\tau(q)-\varepsilon}_\mu(E_i)\geq
\frac{1}{C_\ball}\sum_i \nu^\ast(E_i)\geq
\frac{1}{C_\ball}\nu^\ast(S_\mu),$$ which implies
$$\mathscr{P}^{q,\tau(q)-\varepsilon}_\mu(S_\mu)\geq
\frac{1}{C_\ball}\nu^\ast(S_\mu)>0.$$

To prove the second inequality, one notices that for
$\varepsilon>0$, there exists a subsequence $\{n_k\}$ such that
$\tau_{n_k}(q)<\underline{\tau}(q)+\varepsilon$, for every $k\geq
1$. Take any subset $F\subset S_\mu$, and we choose the natural
centered $3^{-n_k}$-covering of $F$, which is a set of elements
belonging in $\mathscr{A}^{n_k}$. Now
$$\overline{\mathscr{H}}^{q,\underline{\tau}(q)+\varepsilon}_{\mu,3^{-n_k}}(F)\leq
\sum_{w\in \mathscr{A}^{n_k}} \mu(w)^q
3^{-n_k(\underline{\tau}(q)+\varepsilon)}\leq \sum_{w\in
  \mathscr{A}^{n_k}} \mu(w)^q 3^{-n_k\tau_{n_k}(q)}=1,$$
which means
$$\overline{\mathscr{H}}^{q,\underline{\tau}(q)+\varepsilon}_\mu(F)\leq
1,$$
and
$$\mathscr{H}^{q,\underline{\tau}(q)+\varepsilon}_\mu(S_\mu)\leq 1.$$

To prove the last inequality, it is sufficient to show that
$\overline{\mathscr{H}}^{q,\underline{\tau}(q)}_\mu(S_\mu)> 0$. Let
$m\geq 1$, for any $3^{-m}$-covering $\mathscr{B}_m=\{w_i\}_i$ of
$S_\mu$, we have
\begin{eqnarray}
\sum_{w_i\in \mathscr{B}_m} \mu(w_i)^q
3^{-|w_i|\underline{\tau}(q)}&=&\sum_{n\geq m}\sum_{w_i\in
\mathscr{A}^n,\, w_i\in \mathscr{B}_m}\mu(w_i)^q
3^{-n\underline{\tau}(q)}\nonumber\\
&\geq & \sum_{n\geq m}\sum_{w_i\in \mathscr{A}^n,\, w_i\in
\mathscr{B}_m} \nu(w)\nonumber\\
&=&\sum_{w_i\in
\mathscr{B}_m}\nu(w_i)\geq\nu^\ast(S_\mu)\nonumber.
\end{eqnarray}
This implies
$$\overline{\mathscr{H}}^{q,\underline{\tau}(q)}_{\mu,3^{-m}}
(S_\mu)\geq\nu^\ast(S_\mu),$$
which yields
$$\overline{\mathscr{H}}^{q,\underline{\tau}(q)}_\mu(S_\mu)>0.$$
So the proof is finished.

\end{proof}

\begin{remark}\label{r1}
  In fact, here one easily checks that in the symbolic space,
  $\overline{\mathscr{H}}$ and $\mathscr{H}$ are the same. So in the
  proof of the second inequality, one does not need to introduce the
  subset $F$.
\end{remark}

Now we are back to compute the functions $B$ and $b$. To obtain the
equalities $B(q)=\tau(q)$, $b(q)=\underline{\tau}(q)$, it is now
sufficient to prove that $\tau(q)\leq B(q)$, $\underline{\tau}(q)\leq
b(q)$ and $\underline{\tau}(q)\geq b(q)$, which are just consequences
of Lemma \ref{l7}.

So Theorem \ref{t5} has been done for $c_1=c_2=3$. In general case,
one can compute
$$\tau_n(q)=\frac{\frac{N_n}{n}\log(a_1^q+\cdots+a_{c_1}^q)}
{\frac{N_n}{n}\log c_1+(1-\frac{N_n}{n})\log c_2}+
\frac{(1-\frac{N_n}{n})\log(b_1^q+\cdots+b_{c_2}^q)}
{\frac{N_n}{n}\log c_1+(1-\frac{N_n}{n})\log c_2},$$
which follows
$$\tau(q)=\sup\{\log_{c_1}(a_1^q+\cdots+a_{c_1}^q),\log_{c_2}(b_1^q+
\cdots+b_{c_2}^q)\},$$
$$\underline{\tau}(q)=\inf\{\log_{c_1}(a_1^q+\cdots+a_{c_1}^q),
\log_{c_2}(b_1^q+\cdots+b_{c_2}^q)\}.$$

Using the very same method as above, and replacing the metric $3^{-n}$
by $c_1^{-N_n}c_2^{-(n-N_n)}$, one checks that Theorem \ref{t5} is
valid.

\section{Measures with analytic Olsen's functions}

The purpose of this section is to find a measure whose Olsen's
functions $B$ and $b$ both are analytic and their graphs are tangent
to each other at two special points, which means, the graphs of the
two functions have four intersections, counted with their orders. So
we recall some results about the number of zeros of generalized
Dirichlet polynomial by G.J.O. Jameson.

\begin{definition}
Let $f:\mathbb{R}\rightarrow\mathbb{R}$ be an analytic function. $x_0$
is called a zero of $f$ of order $k$ ($k\geq 0$) if
$$f(x_0)=f'(x_0)=\cdots=f^{(k-1)}(x_0)=0, \textrm{and }
f^{(k)}(x_0)\neq 0.$$

Let $f_1,f_2:\mathbb{R}\rightarrow\mathbb{R}$ be two analytic
functions. $(x_0,y_0)$ is called an intersection of the graphs of
$f_1$ and $f_2$, of order $k$ ($k\geq 0$), if $x_0$ is the zero of
function $f_1-f_2$ of order $k$, and $y_0=f_1(x_0)$.
\end{definition}

\begin{definition}
A (generalized) Dirichlet polynomial is a function of the form
$$F(x)=\sum_{j=1}^n a_j e^{p_j x}, x\in \mathbb{R},$$ where the $p_j$
can be any real numbers (listed in descending order).

The length of a Dirichlet polynomial is the number of non-zero terms
in its defining expression.

Among Dirichlet polynomials, a special case is when each $a_j$ is
either 1 or -1, with equally many of each occurring. We call this type
bipartite.
\end{definition}
With the same notations as above, Jameson proved the next theorem:

\begin{thm}\label{t8}  \emph{(see \cite{Jam})}
If $(a_j)$ is bipartite of length $2n$, then the number of zeros of
$F$ (counted with their orders) is not greater than $n$.
\end{thm}

\begin{cor}\label{c9}
The Olsen's functions $B$ and $b$ of the measure
$\mu$ in Theorem \ref{t5}, when $c_1=c_2=3$, cannot be analytic.

\end{cor}

\begin{proof}
Denote by
$$\theta_1(q)=\log_3(a_1^q+a_2^q+a_3^q),$$
$$\theta_2(q)=\log_3(b_1^q+b_2^q+b_3^q).$$

The number of intersections of the functions $\theta_1$ and $\theta_2$
is the same as the number of zeros of the Dirichlet polynomial
$$F(q)=a_1^q+a_2^q+a_3^q-(b_1^q+b_2^q+b_3^q).$$ But this is a
bipartite of length 6, so it has at most 3 zeros, counted with their
orders. Since 0 and 1 are already two zeros of $F$, the graphs of
$\theta_1$ and $\theta_2$ cannot be tangent at both (0,1) and (1,0)
(or $F$ will have at least four zeros!). So at least one point, say
(1,0), is an intersection of order 1 of the functions $\theta_1$ and
$\theta_2$, and the graphs get crossed at (1,0) with different
derivatives, which implies $b$ and $B$ cannot be analytic, and their
graphs cannot be tangent to each other.

\end{proof}

However, when $c_1=c_2=4$, for Olsen functions $B$ and $b$ of the
measure $\mu$, the corresponding Dirichlet polynomial
$$F(q)=a_1^q+a_2^q+a_3^q+a_4^q-(b_1^q+b_2^q+b_3^q+b_4^q)$$ will be a
bipartite of length 8, which means, $F$ has at most 4 zeros according
to Jameson. So it is not impossible that 0 and 1 are both zeros of
order 2. And this is exactly the condition we are looking for.  The
object of next theorem is to find two such proper groups of parameters
$(a_j)$ and $(b_j)$.

Denote
$$\theta(x_1,x_2,x_3,x_4;q)=\log(x_1^q+x_2^q+x_3^q+x_4^q),$$ then one
easily computes
$$\theta '(x_1,x_2,x_3,x_4;q)=\frac{x_1^q\log x_1+x_2^q\log
  x_2+x_3^q\log x_3+x_4^q\log x_4}{x_1^q+x_2^q+x_3^q+x_4^q},$$ where
$\theta '$ stands for the derivative of the function $\theta$ with
respect to $q$.

\begin{thm}\label{t10}
There exist two different groups $(a_j)$ and $(b_j)$ ($a_j,b_j\in
(0,1)$, $j=1,2,3,4$) such that $\sum a_j=\sum b_j=1$ and
    $$\theta(a_1,a_2,a_3,a_4;0)=\theta(b_1,b_2,b_3,b_4;0),$$
    $$\theta(a_1,a_2,a_3,a_4;1)=\theta(b_1,b_2,b_3,b_4;1),$$
    $$\theta '(a_1,a_2,a_3,a_4;0)=\theta '(b_1,b_2,b_3,b_4;0),$$
    $$\theta '(a_1,a_2,a_3,a_4;1)=\theta '(b_1,b_2,b_3,b_4;1).$$

\end{thm}

\begin{proof}
Only the last two equalities are to be proved. Take
$$(a_1,a_2,a_3,a_4)=(a+t,b,c,d-t),$$
$$(b_1,b_2,b_3,b_4)=(a+u,b+v,c+w,d-(u+v+w)),$$
where $a,b,c,d$ are positive constants to be fixed later on subject to
the condition $a+b+c+d=1$, and $t,u,v,w$ are real numbers. Define
\begin{multline*}
\varphi(t,u,v,w)=(a+t)\log(a+t)+b\log b+c\log c+(d-t)\log
(d-t)-\\
(a+u)\log(a+u)-(b+v)\log(b+v)-(c+w)\log(c+w)-\\
(d-(u+v+w))\log(d-(u+v+w)),
\end{multline*}
\begin{multline*}
\psi(t,u,v,w)=\log(a+t)+\log b+\log
c+\log(d-t)-\log(a+u)-\\
\log(b+v)-\log(c+w)-\log(d-(u+v+w)).
\end{multline*}

Notice that they are just the differences of derivatives of function
$\theta$ (with respect to $(a_j)$ and $(b_j)$) at points 1 and 0. We
wish to find suitable small non-zero numbers $t,u,v,w$ such that both
of the functions $\varphi$ and $\psi$ vanish.

It is easy to get
$$\left\{
  \begin{array}{ll}
    \frac{\partial\varphi}{\partial
      t}=\log\frac{a+t}{d-t}\\ \frac{\partial\varphi}{\partial
      u}=\log\frac{d-(u+v+w)}{a+u}\\ \frac{\partial\varphi}{\partial
      v}=\log\frac{d-(u+v+w)}{b+v}\\ \frac{\partial\varphi}{\partial
      w}=\log\frac{d-(u+v+w)}{c+w}
  \end{array}
\right.$$
and
$$\left\{
  \begin{array}{ll}
    \frac{\partial\psi}{\partial
      t}=\frac{1}{a+t}-\frac{1}{d-t}\\ \frac{\partial\psi}{\partial
      u}=\frac{1}{d-(u+v+w)}-\frac{1}{a+u}\\ \frac{\partial\psi}{\partial
      v}=\frac{1}{d-(u+v+w)}-\frac{1}{b+v}\\ \frac{\partial\psi}{\partial
      w}=\frac{1}{d-(u+v+w)}-\frac{1}{c+w}
  \end{array}
\right.$$

The Jacobian matrix at point (0,0,0,0) is
$$\frac{\partial(\varphi,\psi)}{\partial(t,u,v,w)}|_{(0,0,0,0)}=
\begin{pmatrix}
\log\frac{a}{d}&\log\frac{d}{a}&\log\frac{d}{b}&\log\frac{d}{c}\\
\frac{1}{a}-\frac{1}{d}&\frac{1}{d}-\frac{1}{a}&
\frac{1}{d}-\frac{1}{b}&\frac{1}{d}-\frac{1}{c}
\end{pmatrix}\cdots\cdots(\ast)$$

Take out the two middle columns and one can expect the determinant of
this submatrix nonzero. In fact, just let
$$(a,b,c,d)=\left(\frac{1}{10},\frac{2}{10},\frac{3}{10},
\frac{4}{10}\right),$$ then
$$\frac{\partial(\varphi,\psi)}{\partial(u,v)}\Big|_{(0,0)}=
\left( \begin{array}{cc}\log 4&\log 2\\
    -\frac{15}{2}&-\frac{5}{2}\end{array} \right),$$
which is
a nondegenerate matrix.

At last, noticing that $\varphi(0,0,0,0)=\psi(0,0,0,0)=0$ and
recalling the implicit function theorem, we obtain a map $F$
satisfying that for any small $(t,w)$, we have $(u,v)=F(t,w)$ such
that
$$\varphi(t,u,v,w)=\psi(t,u,v,w)=0.$$
\end{proof}

\begin{remark}\label{r2}
One sees from the matrix $(\ast)$ that the first two columns should
not be taken out, since it is obviously a degenerate matrix. This is
because, when $v$ and $w$ are given, $t$ and $u$ can be never found to
vanish the functions $\varphi$ and $\psi$.

For example, take $v=0$, $w\neq 0$ but very small. Now the two groups
of parameters are $(a+t,b,c,d-t)$ and $(a+u,b,c+w,d-(u+w))$. And we
cannot find suitable $t$ and $u$: otherwise, it is equivalent to
obtain a bipartite of length 6 but with 4 zeros! This just causes a
contradiction.

Besides, any other two columns in the matrix $(\ast)$ are available.
\end{remark}

\begin{remark}\label{r3}
If we take
$$(a,b,c,d)=\left(\frac{1}{9},\frac{2}{9},\frac{2}{9},
\frac{4}{9}\right),$$
then we can still obtain a map $F$ satisfying that for any small
$(t,w)$, we have $(u,v)=F(t,w)$ such
that
$$\varphi(t,u,v,w)=\psi(t,u,v,w)=0.$$

Let $t=0$, $w\neq 0$ but very small. Then
$(a_1,a_2,a_3,a_4)=\left(\frac{1}{9},\frac{2}{9},\frac{2}{9},
\frac{4}{9}\right)$. Denote
$(e_1,e_2)=\left(\frac{1}{3},\frac{2}{3}\right)$, then we have
$$\log_4(a_1^q+a_2^q+a_3^q+a_4^q)=\log_2(e_1^q+e_2^q).$$

\end{remark}

\begin{pro}\label{p11}
  Let $\mathscr{A}=\{0,1,2,3\}$. There exists a probability measure
  $\mu$ on $\partial\mathscr{A}^\ast$ such that its Olsen's functions
  $B$ and $b$ are analytic and their graphs differ except at two points
  where they are tangent, with $B(0)=b(0)$, $B(1)=b(1)$, and $B(q)>b(q)$
  for all $q\neq 0,1$. Moreover $B$ and $b$ are convex and
  $B'(\mathbb{R})$ and $b'(\mathbb{R})$ both are intervals of positive
  length.

At the same time, if we let $\mathscr{A}_1=\{0,1\}$,
$\mathscr{A}_2=\{0,1,2,3\}$ and $(T_k)$ be a suitable sequence of
integers, then the statements above are true for the mixed symbolic
space $\partial\mathscr{A}_{1,2}^\ast$.
\end{pro}

\begin{proof}
Choose $(a_j)$ and $(b_j)$ in Theorem \ref{t10}. By Theorem \ref{t5}
(set $c_1=c_2=4$), there exists a probability measure $\mu$ on
$\partial\mathscr{A}^\ast$ such that
$$B(q)=\sup\{\log_4(a_1^q+a_2^q+a_3^q+a_4^q),
\log_4(b_1^q+b_2^q+b_3^q+b_4^q)\},$$
$$b(q)=\inf\{\log_4(a_1^q+a_2^q+a_3^q+a_4^q),
\log_4(b_1^q+b_2^q+b_3^q+b_4^q)\}.$$

However, all intersections of $\theta(a_1,a_2,a_3,a_4;q)$ and
$\theta(b_1,b_2,b_3,b_4;q)$ are clear: (0,1) and (1,0) are all of the
four intersections, both of order 2. So these two curves are tangent
to each other and one curve is always above the other (the order of
the intersection is even). This follows the first part of the
conclusion.

For the second part, just by using Remark \ref{r3}, we can find
$(a_j),(b_j)$ and $(e_j)$ such that
$$\log_4(a_1^q+a_2^q+a_3^q+a_4^q)=\log_2(e_1^q+e_2^q).$$
So the same conclusion follows.

\end{proof}

\begin{remark}\label{r4}
We know from Theorem \ref{t4} that the measures above cannot satisfy
the classsical multifractal formalism at any point $q\neq 0,1$.
\end{remark}

\section{Interpretation of Legendre transforms of $b$ and $B$}

In this section, we are to compute the dimensions of the level sets
of local H\"{o}lder exponents of the measure $\mu$ on
$\partial\mathscr{A}^\ast$ with analytic dimension functions (see
Proposition \ref{p11}).

We can see from the construction of measure $\mu$ that
$0<a_1<a_2<a_3<a_4<1$ and $0<b_1<b_2<b_3<b_4<1$. Fix such two groups
of parameters and let us present an analytic result first.

\begin{lemma}\label{l12}
For any $\alpha\in (-\log_4 a_4,-\log_4 a_1)$, there exist
$\tilde{a}_1,\,\tilde{a}_2,\,\tilde{a}_3,\,\tilde{a}_4\in (0,1)$ such that
$\sum_{i=1}^4 \tilde{a}_i=1$ and
$$-\sum_{i=1}^4 \tilde{a}_i\log_4 a_i=\alpha,$$
$$\frac{\log \frac{\tilde{a}_2}{\tilde{a}_1}}{\log
  \frac{a_2}{a_1}}=\frac{\log \frac{\tilde{a}_3}{\tilde{a}_1}}{\log
  \frac{a_3}{a_1}}=\frac{\log \frac{\tilde{a}_4}{\tilde{a}_1}}{\log
  \frac{a_4}{a_1}}.$$

\end{lemma}

\begin{proof}
Assume that $\tilde{a}_1,\tilde{a}_2,\tilde{a}_3,\tilde{a}_4\in (0,1)$
such that $\sum_{i=1}^4 \tilde{a}_i=1$.

For a given $q$, to insure that
$$\frac{\log \frac{\tilde{a}_2}{\tilde{a}_1}}{\log
  \frac{a_2}{a_1}}=\frac{\log \frac{\tilde{a}_3}{\tilde{a}_1}}{\log
  \frac{a_3}{a_1}}=\frac{\log \frac{\tilde{a}_4}{\tilde{a}_1}}{\log
  \frac{a_4}{a_1}}=q,$$
one has
$$\tilde{a}_i=\tilde{a}_1(\frac{a_i}{a_1})^q,\,i=2,3,4.$$

So
$$\sum_{i=1}^4 \tilde{a}_i=\tilde{a}_1(1+\sum_{i=2}^4
(\frac{a_i}{a_1})^q)=1,$$ which implies
$$\tilde{a}_i=\frac{a_i^q}{\sum_{j=1}^4 a_j^q},\,i=1,2,3,4.$$

One sees that $q$ can take any value in $(-\infty,+\infty)$. Now,
consider the decreasing and convex function
$$\theta(q)=\log\sum_{j=1}^4 a_j^q,$$
then
$$-\sum_{i=1}^4 \tilde{a}_i\log_4 a_i=-\sum_{i=1}^4
\frac{a_i^q}{\sum_{j=1}^4 a_j^q} \log_4 a_i=-\frac{\sum_{i=1}^4
  a_i^q\log_4 a_i}{\sum_{j=1}^4 a_j^q}=-\theta'(q),$$
which reaches any value in $(-\log_4 a_4,-\log_4 a_1)$.

\end{proof}

Without losing generality, we may assume that $a_1<b_1$, then we
have
$$B(q)=\log_4(a_1^q+a_2^q+a_3^q+a_4^q),$$
$$b(q)=\log_4(b_1^q+b_2^q+b_3^q+b_4^q).$$

\begin{thm}\label{t13}
For any $\alpha\in (-\log_4 b_4,-\log_4 b_1)$, we have
$$\emph{dim} X(\alpha)=b^\ast(\alpha),$$
$$\emph{Dim} X(\alpha)=B^\ast(\alpha).$$

\end{thm}

\begin{proof}
By Lemma \ref{l12}, we can construct a new probability measure $\nu$ on the
tree just as $\mu$, but replacing $\{a_i,b_i\}$ with
$\{\tilde{a}_i,\tilde{b}_i\}$ such that
$$\sum_{i=1}^4 \tilde{a}_i=\sum_{i=1}^4 \tilde{b}_i=1,$$
$$-\sum_{i=1}^4 \tilde{a}_i\log_4 a_i=-\sum_{i=1}^4 \tilde{b}_i\log_4
b_i,$$
$$\frac{\log \frac{\tilde{a}_2}{\tilde{a}_1}}{\log
  \frac{a_2}{a_1}}=\frac{\log \frac{\tilde{a}_3}{\tilde{a}_1}}{\log
  \frac{a_3}{a_1}}=\frac{\log \frac{\tilde{a}_4}{\tilde{a}_1}}{\log
  \frac{a_4}{a_1}},$$
$$\frac{\log \frac{\tilde{b}_2}{\tilde{b}_1}}{\log
  \frac{b_2}{b_1}}=\frac{\log \frac{\tilde{b}_3}{\tilde{b}_1}}{\log
  \frac{b_3}{b_1}}=\frac{\log \frac{\tilde{b}_4}{\tilde{b}_1}}{\log
  \frac{b_4}{b_1}}.$$

Fix $\lambda< 1$ and define
\begin{multline*}
\varphi_\nu(x)=\limsup_{\delta\rightarrow 0}\frac{-1}{\log
  \delta}\log\sup\left\{\sum_i
\mu(B_i)^x\nu(B_i)\ :\ \right.\\ \left.(B_i)_i \textrm{ packing of
}S_\mu\ \textrm{with }\lambda\delta<r_i\leq\delta\vphantom
   {\sum_i}\right\}
\end{multline*}

Then it is easy to compute
$$\varphi_\nu(x)=\log_4 \max\left\{\sum_{i=1}^4 a_i^x
\tilde{a}_i,\sum_{i=1}^4 b_i^x \tilde{b}_i\right\}.$$

The method of choosing $\{\tilde{a}_i,\tilde{b}_i\}$ insures that
$\varphi_\nu'(0)$ exists, so we define
$$\alpha=\alpha_\nu=-\varphi_\nu'(0)=-\sum_{i=1}^4 \tilde{a}_i\log_4
a_i=-\sum_{i=1}^4 \tilde{b}_i\log_4 b_i.$$

It is obvious that $\alpha\in (-\log_4 b_4,-\log_4 b_1)$. On the other
hand, for any $\beta\in (-\log_4 b_4,-\log_4 b_1)$, there exist
$\tilde{a}_i,\tilde{b}_i$ such that $\alpha_\nu=\beta$, where $\nu$ is
the corresponding measure with respect to
$\tilde{a}_i,\tilde{b}_i$. So $\alpha$ can take any value in the
interval $(-\log_4 b_4,-\log_4 b_1)$.

Now we estimate the bounds of the dimensions of the level sets. The
strong law of large numbers shows that
$$\liminf \frac{\log_4
  \nu(B(x,4^{-n}))}{-n}=\min\{h(\tilde{a}),h(\tilde{b})\},$$
$$\limsup \frac{\log_4
  \nu(B(x,4^{-n}))}{-n}=\max\{h(\tilde{a}),h(\tilde{b})\},$$
for $\nu$-almost every $x$, where
$$h(\tilde{a})=-\sum_{i=1}^4 \tilde{a}_i\log_4 \tilde{a}_i,$$
$$h(\tilde{b})=-\sum_{i=1}^4 \tilde{b}_i\log_4 \tilde{b}_i.$$

So it deduces from Lemma \ref{l3} that
$$\dim X(\alpha)\geq \min\{h(\tilde{a}),h(\tilde{b})\},$$
$$\Dim X(\alpha)\geq \max\{h(\tilde{a}),h(\tilde{b})\}.$$

To compute $h(\tilde{a})$ and $h(\tilde{b})$, set
$$q_a=\frac{\log \frac{\tilde{a}_2}{\tilde{a}_1}}{\log
  \frac{a_2}{a_1}}=\frac{\log \frac{\tilde{a}_3}{\tilde{a}_1}}{\log
  \frac{a_3}{a_1}}=\frac{\log \frac{\tilde{a}_4}{\tilde{a}_1}}{\log
  \frac{a_4}{a_1}},$$
then
\begin{eqnarray}
B'(q_a)&=&\frac{a_1^{q_a}\log_4 a_1+a_2^{q_a}\log_4 a_2+a_3^{q_a}\log
  a_3+a_4^{q_a}\log
  a_4}{a_1^{q_a}+a_2^{q_a}+a_3^{q_a}+a_4^{q_a}}\nonumber\\
&=&\frac{\log_4 a_1+\frac{a_2^{q_a}}{a_1^{q_a}}\log_4
  a_2+\frac{a_3^{q_a}}{a_1^{q_a}}\log_4
  a_3+\frac{a_4^{q_a}}{a_1^{q_a}}\log_4
  a_4}{1+\frac{a_2^{q_a}}{a_1^{q_a}}+
  \frac{a_3^{q_a}}{a_1^{q_a}}+\frac{a_4^{q_a}}{a_1^{q_a}}}\nonumber\\
&=&\frac{\log_4 a_1+\frac{\tilde{a}_2}{\tilde{a}_1}\log_4
  a_2+\frac{\tilde{a}_3}{\tilde{a}_1}\log_4
  a_3+\frac{\tilde{a}_4}{\tilde{a}_1}\log_4
  a_4}{1+\frac{\tilde{a}_2}{\tilde{a}_1}+
  \frac{\tilde{a}_3}{\tilde{a}_1}+\frac{\tilde{a}_4}{\tilde{a}_1}}\nonumber\\
&=&\tilde{a}_1
\log_4 a_1+\tilde{a}_2 \log_4 a_2+\tilde{a}_3 \log_4 a_3+\tilde{a}_4
\log_4 a_4\nonumber\\
&=&\varphi_\nu'(0)=-\alpha\nonumber.
\end{eqnarray}
Moreover,
\begin{eqnarray}
B(q_a)+q_a \alpha
&=&\log_4(a_1^{q_a}+a_2^{q_a}+a_3^{q_a}+a_4^{q_a})+q_a\alpha\nonumber\\
&=&\log_4 a_1^{q_a}(1+\frac{a_2^{q_a}}{a_1^{q_a}}+
\frac{a_3^{q_a}}{a_1^{q_a}}+\frac{a_4^{q_a}}{a_1^{q_a}})+q_a\alpha\nonumber\\
&=& q_a\log_4 a_1+\log_4(1+\frac{\tilde{a}_2}{\tilde{a}_1}+
\frac{\tilde{a}_3}{\tilde{a}_1}+
\frac{\tilde{a}_4}{\tilde{a}_1})+q_a\alpha\nonumber\\
&=& q_a\log_4
a_1-\log_4 \tilde{a}_1-q_a(\tilde{a}_1 \log_4 a_1+\cdots+\tilde{a}_4 \log_4
a_4)\nonumber\\
&=& -\log_4 \tilde{a}_1+\tilde{a}_2\log_4\frac{a_1^{q_a}}{a_2^{q_a}}+
\tilde{a}_3\log_4\frac{a_1^{q_a}}{a_3^{q_a}}+
\tilde{a}_4\log_4\frac{a_1^{q_a}}{a_4^{q_a}}\nonumber\\
&=& -\log_4
\tilde{a}_1+\tilde{a}_2\log_4\frac{\tilde{a}_1}{\tilde{a}_2}+
\tilde{a}_3\log_4\frac{\tilde{a}_1}{\tilde{a}_3}+
\tilde{a}_4\log_4\frac{\tilde{a}_1}{\tilde{a}_4}\nonumber\\
&=& -\tilde{a}_1\log_4\tilde{a}_1-\tilde{a}_2\log_4\tilde{a}_2-
\tilde{a}_3\log_4\tilde{a}_3-\tilde{a}_4\log_4\tilde{a}_4\nonumber\\
&=& h(\tilde{a})\nonumber.
\end{eqnarray}
And set
$$q_b=\frac{\log \frac{\tilde{b}_2}{\tilde{b}_1}}{\log
  \frac{b_2}{b_1}}=\frac{\log \frac{\tilde{b}_3}{\tilde{b}_1}}{\log
  \frac{b_3}{b_1}}=\frac{\log \frac{\tilde{b}_4}{\tilde{b}_1}}{\log
  \frac{b_4}{b_1}},$$
with the very same arguments, we have
$$b'(q_b)=-\alpha,$$
$$b(q_b)+q_b \alpha=h(\tilde{b}).$$
Thus
$$h(\tilde{a})=B(q_a)+q_a \alpha=B(q_a)-q_a B'(q_a)=B^\ast(-B'(q_a)),$$
$$h(\tilde{b})=b(q_b)+q_b \alpha=b(q_b)-q_b b'(q_b)=b^\ast(-b'(q_b)),$$
which gives the lower bounds of the dimensions of the level sets:
\begin{eqnarray*}
\dim X(\alpha) &\geq& b^\ast(-b'(q_b))=b^\ast(\alpha),\\
\Dim X(\alpha) &\geq& B^\ast(-B'(q_a))=B^\ast(\alpha).
\end{eqnarray*}

But we have the opposite inequalities according to Theorem \ref{t2}, so
finally we get
\begin{eqnarray*}
\dim X(\alpha) &=& b^\ast(\alpha),\\
\Dim X(\alpha) &=& B^\ast(\alpha).
\end{eqnarray*}
\end{proof}

\section{Measures onto the real line}

\subsection{Generalized Gray codes}{\ }

Let $\mathscr{A}=\{0,1,\cdots,c-1\}$. There is a natural way to enumerate the $n$-cylinders
of $\partial\mathscr{A}^\ast$: for $w = \varepsilon_1 \varepsilon_2\cdots\varepsilon_n$, we set
$$\iota(w) = \sum_{j=0}^{n-1} \varepsilon_{n-j}c^j,$$
but we need other orderings of the cylinders.

We define a transformation $g$ on $\mathscr{A}^\ast$, such that
$g(j)=j$ for all $j\in {\mathscr A}$ and such that, for
any $w=\varepsilon_1 \varepsilon_2\cdots\varepsilon_n\in \mathscr{A}^\ast$,
$$g(\varepsilon_1 \varepsilon_2\cdots\varepsilon_n)=g(\varepsilon_1 \varepsilon_2\cdots\varepsilon_{n-1})\cdot k\,\,(n\geq 2),$$
where $k\in \mathscr{A}$ and $k\equiv (\varepsilon_n-\varepsilon_{n-1})\pmod{c}$.

We can see from the definition that if $w,v\in \mathscr{A}^n$ are two
contiguous words under natural enumeration, then $g(w)$ and $g(v)$
differ by one digit exactly. This property extends the one of the
classical Gray code \cite{Gra}, which is the one we obtain when $c=2$.

Also, it is obvious that $v\prec w$ implies $g(v)\prec g(w)$. So $g$
induces a transformation on $\partial\mathscr{A}^\ast$, still denoted
by $g$: for any $x=\varepsilon_1 \varepsilon_2\cdots\varepsilon_n\cdots $,
$g(x)$ is the unique element of $\bigcap_{n\geq 1}[g(\varepsilon_1
\varepsilon_2\cdots\varepsilon_n)].$ It is easy to see that $g$ is an isometry.

\begin{remark}\label{r5}
  When $c\geq 3$, the way to define a Gray code is not unique. In
  fact, let $\widetilde{g}(j) = j$ for $j\in {\mathscr A}$, and, for
  $w\ne \epsilon$ and $j\in {\mathscr A}$,
$$\widetilde{g}(w\cdot j)=\widetilde{g}(w)\cdot k,$$
where $k=c-1-j$ when $\iota(w)$ is odd, and $k=j$ when $\iota(w)$ is
even. Then $\widetilde{g}$ has the same properties as $g$.
\end{remark}

\subsection{Measures on $[0,1]$}{\ }

In this section, we work on $[0,1]$ and try to get the same conclusion
as Proposition \ref{p11}, i.e. to find a probability measure on
$[0,1]$ such that its Olsen's functions $B$ and $b$ are analytic and
their graphs differ except at two points where they are tangent, with
$B(0)=b(0)$, $B(1)=b(1)$, and $B(q)>b(q)$ for all $q\neq 0,1$.

There is a natural map $\gamma$ from $\partial\mathscr{A}^\ast$ onto
$[0,1]$: $$\text{for\quad} x=\varepsilon_1 \varepsilon_2\cdots\varepsilon_n
\cdots\in \partial\mathscr{A}^\ast ,\quad \gamma(x) = \sum_{n\geq
  1} \varepsilon_n c^{-n}.$$
This map sends cylinders onto $c$-adic intervals; more precisely, if
$w\in \mathscr{A}^n$, the image of $[w]$ under $\gamma$ is the interval
$[\iota(w)c^{-n},(\iota(w)+1)c^{-n}]$.

Now if $\mu$ is a  measure considered in Theorem \ref{t5}, with the
restriction $c_1=c_2=c$, the measure $\nu$ image of $\mu$ under
$\gamma\comp g^{-1}$ (i.e., $\nu(E)=\mu(g\comp\gamma^{-1}(E))$ for
any Borel set $E\subset [0,1]$) is doubling,  because of the
properties of the Gray code $g$.

Since $\nu$ is doubling it is well known that to perform its
multifractal analysis one can use coverings and packings with $c$-adic
intervals as well as coverings and packings with genaral intervals.

So, since $c$-adic intervals correspond to cylinders, $\mu$ and $\nu$
share the same multifractal analysis.

In particular, if we use the measure in Theorem \ref{t13}, the $b$ and
$B$ functions for $\nu$ are analytic, their graphs have intersections of
order 2 at points (1,0) and (0,1), and, for $q\notin \{0,1\}$, $b(q)<
B(q)$. Moreover the Legendre transforms of these functions give the
Hausdorff and packing dimensions of the level sets of
local H\"{o}lder exponents of the measure $\nu$.

\flushleft {\bf Acknowledgements}

The author is grateful to professor Jacques Peyri{\`e}re for his
guidance and valuable comments.

\end{document}